\documentclass[12pt,reqno]{amsart}
\usepackage{amsmath,amssymb,amsthm, amsfonts}
\usepackage{color}
\usepackage[breaklinks=true]{hyperref}
\usepackage{enumitem}
\usepackage{graphicx}

\theoremstyle{plain}
\newtheorem{theorem}{Theorem}[section]

\newtheorem{lemma}[theorem]{Lemma}

\newtheorem{conjecture}{Conjecture}[section]

\theoremstyle{definition}
\newtheorem{definition}[theorem]{Definition}

\newtheorem{remark}[theorem]{Remark}

\numberwithin{equation}{section}
\numberwithin{table}{section}
\allowdisplaybreaks
\newcommand{\Gegen}[1]{\tilde{C}_{#1}^{(\alpha)}} 

\newcommand{\hilbert}{\mathcal{H}}

\begin{document}
\title[Schoenberg-Delsarte thresholding]{On positive definite thresholding of correlation matrices}

\author{Sujit Sakharam Damase}
\address[S.S.~Damase]{Department of Mathematics, Indian Institute of
Science, Bangalore 560012, India;}
\email{\tt sujits@iisc.ac.in}

\author{James Eldred Pascoe}
\address[J.E.~Pascoe]{Department of Mathematics, Drexel University, Philadelphia, PA 19104, United States;}
\email{\tt jep362@drexel.edu}
\thanks{We thank J C Bose Fellowship JCB/2021/000041 of SERB (PI: Tirthankar Bhattacharyya) for supporting this research.}

\date{\today}



\begin{abstract}
    Standard thresholding techniques for correlation matrices often destroy positive semidefiniteness. We investigate the construction of positive definite functions that vanish on specific sets $K \subseteq [-1,1)$, ensuring that the thresholded matrix remains a valid correlation matrix. We establish existence results, define a criterion for faithfulness based on the linear coefficient of the normalized Gegenbauer expansion in analogy with Delsarte's method in coding theory, and provide bounds for thresholding at single points and pairs of points.
    
    We prove that for correlation matrices of rank $n$, any soft-thresholding operator that preserves positive semidefiniteness necessarily induces a geometric collapse of the feature space, as quantified by an $\mathcal{O}(1/n)$ bound on the faithfulness constant. Such demonstrates that geometrically unbiased soft-thresholding limits the recoverable signal.
\end{abstract}
\maketitle
\section{Introduction}
    Given an observed correlation matrix $M=(m_{ij})_{ij}$ from some sample, one may want to assume small entries are actually indeed zero---that is, the underlying random variables are independent. In high-dimensional statistics, thresholding estimators are canonical tools for covariance and correlation matrix regularization, particularly when the number of features exceeds the sample size \cite{bickel2008covariance, rothman2009generalized, cai2011adaptive}. 
    
    The {\bf hard thresholding} technique is to consider $f[M] = (f(m_{ij}))_{ij}$ where $f(x) = x$ if $|x| \geq \varepsilon$ and $f(x)=0$ otherwise.
    The {\bf soft thresholding} technique considers $f[M] = (f(m_{ij}))_{ij}$ where $f$ is a continuous function vanishing on $|x|\leq \varepsilon$. 
    
    A central, notorious issue in this literature is that $f[M]$ is generally not positive semidefinite, and therefore, it cannot be a valid correlation matrix. While statistical practice often relies on post-hoc eigenvalue clipping or assuming strict structural sparsity to avoid this \cite{rothman2009generalized}, we investigate the hard algebraic limits of finding thresholding functions that intrinsically preserve positive definiteness.


    \begin{definition}\label{1.1}
    We say $f: [-1,1] \rightarrow \mathbb{R}$ is {\bf positive definite on $S^{n-1}$} if, for all $N$ and all $x_{1},\cdots, x_{N} \in S^{n-1}$, the matrix $f[M]$ is a correlation matrix, where $M= (\langle x_{i}, x_{j} \rangle)_{i,j=1}^{N}$ is the correlation matrix of the vectors $x_{i}.$
    \end{definition}

    Such maps were classified by Schoenberg in the sphere case and by Bochner in the general case of compact two-point homogeneous spaces. To properly analyze these functions in the context of probability, we utilize the \textbf{normalized Gegenbauer polynomials} $\Gegen{k}(t) = C_k^{(\alpha)}(t) / C_k^{(\alpha)}(1)$, where $\alpha = (n-2)/2$. These polynomials satisfy $\Gegen{k}(1) = 1$. Schoenberg proved that, a continuous function $f$ is positive definite on $S^{n-1}$ if and only if it admits an expansion:
    $$ f(t) = \sum_{k=0}^\infty a_k \Gegen{k}(t), \quad a_k \geq 0. $$
    If we further require $f(1)=1$ (preserving the diagonal of the correlation matrix), then $\sum_{k\geq 0} a_k = 1$. 

    { \begin{remark}\label{R1.2}
Schoenberg's result can be interpreted through the lens of correlation matrix regularization. Note that Gram matrices of vectors from the unit sphere $S^{n-1}$ are precisely the same as positive semidefinite matrices with diagonal entries $1$ -- i.e., correlation matrices. Moreover, regardless of the number of vectors $x_i$ chosen, the dimension of the sphere constrains the \textit{rank} of the correlation matrix, to at most $n$. Thus, Definition \ref{1.1} now says that a function $f : [-1,1] \to \mathbb{R}$ is positive definite on $S^{n-1}$ if and only if applying $f[-]$ entrywise to correlation matrices of any size but rank $\leq n$, yields $f(1)$ times a(nother) correlation matrix.
\end{remark}

}

    Thus, given a compact set $K\subseteq [-1,1)$ which we want to threshold to $0,$ we seek to find positive definite functions which vanish on $K.$
    We establish their existence based on methods of \cite{gafa} developed for Delsarte analysis of the measurable chromatic number of the unit distance graph on the sphere -- see Theorem \ref{existenceresult}.
    We give a criterion for evaluation inspired by Delsarte theory and based on maximizing the relative weight of the linear term in the Gegenbauer expansion, the {\bf faithfulness constant.}
    Given existence, by recent results of \cite{damase2026complete}, for finite $K$ one can choose the positive definite function to be a polynomial. In contrast to the unconstrained rank case of soft thresholding considered in \cite{guillot2012retaining} where there is a paucity of thresholders available, our fixed rank case provides plenitudes. However, such thresholders must significantly degrade the signal in the geometric sense established by Theorem \ref{degrade} and by Theorem \ref{intthm} -- the latter of which establishes that faithfulness for soft thresholders is $\mathcal{O}(1/n)$ where $n$ is the dimension of the sphere from where the data is drawn (by Remark \ref{R1.2}, this is the maximum possible rank of the correlation matrices obtainable from the data). Fixed rank naturally arises in the regime of low sample high feature data. 

    A common workaround for soft thresholding is to apply Ledoit--Wolf type shrinkage \cite{ledoit2004shrunk} and take some suitable convex combination of a hard (or piece-wise linearly soft) thresholded correlation matrix and the identity matrix to force it into the positive cone, which preserves induced sparsity, as this merely amounts to scaling the off diagonal entries of the correlation matrix. However, Schoenberg's theorem implies that for large generic rank $n$ correlation matrices, the combination asymptotically puts all the weight on the identity matrix if the thresholding function is not a positive definite function.
    
    \section{Background}

\subsection{Kernel embeddings}
The geometric interpretation of positive definite functions relies fundamentally on the construction of reproducing kernel Hilbert spaces (RKHS). 
A symmetric function $K: X \times X \to \mathbb{R}$ on a set $X$ is called a \textbf{positive definite kernel} if for any finite set of points $x_1, \dots, x_N \in X$ and real numbers $c_1, \dots, c_N$, the corresponding Gram matrix is positive semidefinite:
\[ \sum_{i=1}^N \sum_{j=1}^N c_i c_j K(x_i, x_j) \geq 0. \]

The connection between such kernels and Euclidean geometry is established by the Aronszajn Theorem, which guarantees that every valid correlation structure can be realized as a set of inner products in a latent space.

\begin{theorem}[Aronszajn's theorem \cite{paulsen2016introduction}]\label{thm:aronszajn}
    Let $K$ be a positive definite kernel on a set $X$. Then there exists a essentially unique Hilbert space $\mathcal{H}$ and a feature map $\iota: X \to \mathcal{H}$ such that the kernel perfectly recovers the inner product in $\mathcal{H}$:
    \[ \langle \iota(x), \iota(y) \rangle_{\mathcal{H}} = K(x, y) \quad \text{for all } x, y \in X. \]
\end{theorem}

In the context of correlation matrices, the index set $X$ is the unit sphere $S^{n-1}$. A continuous function $f: [-1, 1] \to \mathbb{R}$ is said to be positive definite on $S^{n-1}$ exactly when the composition $K(x, y) = f(\langle x, y \rangle)$ forms a positive definite kernel. 

By Theorem \ref{thm:aronszajn}, applying such a function $f$ to a correlation matrix is geometrically equivalent to embedding the original variables into a new Hilbert space via the map $\iota: S^{n-1} \to \mathcal{H}$. The constraints on the thresholded matrix are thus completely determined by the geometry of this embedding.

\subsection{Spherical Harmonics and Gegenbauer Polynomials}
For a comprehensive treatment of spherical harmonics and orthogonal polynomials, we refer the reader to the great classic \cite{szego}. Let $S^{n-1} \subset \mathbb{R}^n$ denote the unit sphere equipped with the rotationally invariant probability measure $\sigma$. The space $L^2(S^{n-1}, \sigma)$ decomposes into a direct sum of mutually orthogonal subspaces $\mathcal{H}_k$, where $\mathcal{H}_k$ consists of the restriction of harmonic homogeneous polynomials of { total} degree $k$ to the sphere:
\[ L^2(S^{n-1}, \sigma) = \bigoplus_{k=0}^\infty \mathcal{H}_k. \]
The dimension of $\mathcal{H}_k$, denoted $d_k(n)$, is given by:
\[ d_k(n) = \binom{n+k-1}{k} - \binom{n+k-3}{k-2}. \]

The \textbf{Gegenbauer polynomials} arise as the zonal spherical functions associated with this decomposition. In classical literature, these are typically denoted $C_k^{(\alpha)}(t)$, where the parameter is defined as $\alpha = (n-2)/2$. 

One can also see Gegenbauer polynomials as the family of polynomials which are orthogonal with respect to the measure $(1-t^{2})^{\alpha - 1/2}$, this measure arises naturally (up to scaling) as the orthogonal projection of the surface measure from $S^{n-1}$ onto a coordinate axis.

For mathematical clarity when invoking classical identities, we will use $\alpha = (n-2)/2$ to denote the classical parameter, and denote our \textbf{normalized Gegenbauer polynomials} by $\Gegen{k}(t)$, which are scaled such that $\Gegen{k}(1) = 1$.
Explicitly, if $C_k^{(\alpha)}(t)$ are the standard orthogonal polynomials defined by the generating function $(1-2rt+r^2)^{-\alpha} = \sum C_k^{(\alpha)}(t) r^k$, then:
\[ \Gegen{k}(t) = \frac{C_k^{(\alpha)}(t)}{C_k^{(\alpha)}(1)}, \quad \text{where } C_k^{(\alpha)}(1) = \binom{k+2\alpha-1}{k}. \]
These polynomials satisfy the strict bound $|\Gegen{k}(t)| \leq 1$ for all $t \in [-1, 1]$.

\subsection{Darboux estimates}
Darboux's estimate \cite{szego} is useful for developing intuition for Gegenbauer polynomials. Applying the normalization $C_k^{(\alpha)}(1) \sim k^{2\alpha-1}/\Gamma(2\alpha)$ to the standard Darboux expansion, we evaluate the polynomial at $t = \cos \theta$ for $\theta \in (0, \pi)$. As $k \to \infty$:
$$ \Gegen{k}(\cos \theta) = \frac{2^{1-\alpha} \Gamma(2\alpha)}{\Gamma(\alpha)} k^{-\alpha} (\sin \theta)^{-\alpha} \cos\left( (k+\alpha)\theta - \frac{\alpha \pi}{2} \right) + \mathcal{O}\left(k^{-\alpha-1}\right) $$
where $\alpha = \frac{n-2}{2}$. Away from the poles $t = \pm 1$, the normalized polynomials decay uniformly bounded with rate $\mathcal{O}(k^{\frac{2-n}{2}})$. 

\subsection{Derivative Estimates}
To analyze thresholding on intervals (Theorem \ref{intthm}), we require estimates on the derivatives of $\Gegen{k}(t)$ at the origin. 
\begin{lemma}\label{lem:derivative}
    Let $K \subseteq (-1,1)$ compact.
    Then, 
    $\sup_{t\in K}|(\Gegen{k})'(t)|$ is $\mathcal{O}(k^{\frac{4-n}{2}})$.
\end{lemma}
\begin{proof}
    The classical differentiation identity for Gegenbauer polynomials is $\frac{d}{dt} C_k^{(\alpha)}(t) = 2\alpha C_{k-1}^{(\alpha+1)}(t)$ \cite{szego}. Normalizing by $C_k^{(\alpha)}(1)$ and applying the classical dimension parameter $\alpha = \frac{n-2}{2}$ yields:
    \[ \frac{d}{dt} \Gegen{k}(t) = \frac{k(k+n-2)}{n-1} \frac{C_{k-1}^{(\alpha+1)}(t)}{C_{k-1}^{(\alpha+1)}(1)}, \]
    which is $\mathcal{O}(k^{\frac{4-n}{2}})$ by Darboux's estimate.
\end{proof}
\subsection{Three-term recurrence}
To extract structural bounds on the coefficients of thresholding functions, we require the explicit three-term recurrence relation for the normalized Gegenbauer polynomials. 

Unlike the standard orthogonal polynomials, the normalization condition $\Gegen{k}(1) = 1$ forces an asymmetric transition in the degrees. For $k \geq 0$, the recurrence is given by:
$$t \Gegen{k}(t) = c_k \Gegen{k+1}(t) + b_k \Gegen{k-1}(t),$$
where the forward and backward transition weights are given by:
$$c_k = \frac{k+n-2}{2k+n-2}, \quad \text{and} \quad b_k = \frac{k}{2k+n-2}.$$
For $k=0$, we adopt the convention $b_0 = 0$ and $\Gegen{-1}(t) = 0$, which correctly yields $c_0 = 1$. Notice that for all $k$, $c_k + b_k = 1$, preserving the property that the polynomials evaluate to $1$ at the pole $t=1$.
\section{Delsarte's method and faithfulness}
    A {\bf spherical code} $\mathcal{C}$ is a finite subset of the unit sphere $S^{n-1}$. The minimum angle of the code is the smallest angular distance between any distinct pair of points in $\mathcal{C}$. Delsarte's linear programming method \cite{Delsarte} provides an upper bound on the cardinality of such codes using positive definite functions.

    \begin{theorem}[Delsarte's Estimate]
        Let $\theta\in (0, \pi)$ and $f: [-1,1] \to \mathbb{R}$ be a positive definite function on $S^{n-1}$ such that:
        \begin{itemize}
            \item $f(t) \leq 0$ for all $t \in [-1, \cos\theta]$;
            \item The constant coefficient $a_0$ in the Gegenbauer expansion is positive.
        \end{itemize}
        Then, for any spherical code $\mathcal{C}$ with minimum angle $\theta$,
        \[
        |\mathcal{C}| \leq \frac{f(1)}{a_0}.
        \]
    \end{theorem}

    The following proof of Delsarte's estimate gives a roadmap for our investigation of soft thresholding.
    \begin{proof}
    The positive definite function $f$ defines a kernel embedding $\iota: S^{n-1} \to \hilbert$ such that $\langle \iota(x), \iota(y) \rangle_{\hilbert} = f(\langle x, y \rangle)$. 
    We decompose the Hilbert space into orthogonal subspaces of spherical harmonics, $\hilbert = \bigoplus_{k=0}^\infty \hilbert_k$.
    { That is, we view the embedding $\iota$ as a weighted direct sum of the embeddings $\iota_k$ corresponding to the kernels $\Gegen{k}(\langle x,y \rangle).$}
    
    Let $T_0 : \hilbert \to \mathbb{R}$ be the linear operator that extracts the constant component, defined by the action $T_0 \iota(x) = 1$ for any unit vector $x \in S^{n-1}$. 
    The squared norm of the component of $\iota(x)$ in $\hilbert_0$ is given by $a_0 \Gegen{0}(\langle x, x \rangle) = a_0$. Since $T_0$ maps a vector of norm $\sqrt{a_0}$ to the scalar $1$, the operator norm of $T_0$ is precisely:
    \[ \|T_0\| = \frac{1}{\sqrt{a_0}}. \]
    
    Now, let $\mathcal{C} \subset S^{n-1}$ be a spherical code with minimum angle $\theta$. Consider the sum of the embedded points:
    \[ V = \sum_{x \in \mathcal{C}} \iota(x) \in \hilbert. \]
    Applying the operator $T_0$ to this sum yields $T_0(V) = \sum_{x \in \mathcal{C}} 1 = |\mathcal{C}|$. { Squaring, we obtain:
    \begin{equation}\label{Elabel}
        |\mathcal{C}|^2 = |T_0(V)|^{2} \leq  \|T_0\| \|V\| = \frac{1}{a_0} \|V\|^2.
    \end{equation}
    }
    
    We evaluate the squared norm of $V$ directly using the kernel property:
    \[ \|V\|^2 = \sum_{x, y \in \mathcal{C}} \langle \iota(x), \iota(y) \rangle_{\hilbert} = \sum_{x, y \in \mathcal{C}} f(\langle x, y \rangle) = |\mathcal{C}| f(1) + \sum_{x \neq y} f(\langle x, y \rangle). \]
    Since the minimum angle of the code is $\theta$ and $f(t) \leq 0$ for all $t \in [-1, \cos\theta]$, the entire off-diagonal sum is non-positive. This bounds the total energy of the configuration:
    \[ \|V\|^2 \leq |\mathcal{C}| f(1). \]
    
   { Combining this with \eqref{Elabel}} yields the desired bound, $|\mathcal{C}| \leq f(1)/a_0$.
\end{proof}

\begin{remark}
    Note that, the constant coefficient $a_{0}$ of $f$ is the expectation of $f(\langle x ,y\rangle)$ with $x$ and $y$ chosen independently and uniformly at random from $S^{n-1}$.
\end{remark}

\begin{remark}
The proof of Delsarte's estimate is classically presented via algebraic inequalities on the entries of a positive semidefinite Gram matrix. However, the RKHS formulation presented here exposes a deeper structural mechanics native to operator algebra and probability.

By forcing the off-diagonal inner products to be non-positive, we geometrically enforce a prior of mutual independence (or active repulsion) among the embedded random variables. In this framework, the operator $T_0$ does not merely extract a constant; it acts precisely as the canonical \textbf{conditional expectation} projecting the total system down onto the trivial spherical harmonic state (the constant subalgebra). 
 We optimize to find the best unbiased embedding with the largest conditional expectation to the constant function.
In this light, Delsarte's method transcends geometric sphere-packing; it is a strict operator-norm capacity bound on the variance that can survive a conditional expectation under a prior of independence.
\end{remark}
    
    In the classical Delsarte usage (and its adaptation to sphere packing by Cohn and Elkies \cite{cohn-elkies}, and further by Viazovska \cite{viazovska2017sphere,cohn2017sphere}, which uses the analogous theory for $\mathbb{R}^n$), the goal is to maximize the ratio of the total mass $f(1)$ to the constant term $a_0$ (triviality) to obtain the tightest possible bound on code size.

    \begin{remark}
    In our context of thresholding, we invert this perspective. We are not bounding a code size; rather, we are constructing a kernel. We seek to maximize the \textbf{linear coefficient} $a_1$ (which we term faithfulness) subject to the constraint that $f$ vanishes on the thresholding set $K$. Just as Delsarte maximizes triviality to bound density, we maximize faithfulness to minimize information loss.
    Concretely, we have
        \begin{equation}
            a_1\langle x,y\rangle-(1-a_1)\leq  \langle\iota(x),\iota(y) \rangle\leq a_1\langle x,y\rangle+1-a_1.
        \end{equation}
    \end{remark}

\section{Existence and structural bounds}
    We now establish the existence of the desired thresholding functions.
    \begin{theorem}\label{existenceresult}
        Let $K\subseteq [-1,1)$ be compact. There exists a non-zero positive definite function $f$ on $S^{n-1}$ vanishing on $K.$
    \end{theorem}
    \begin{proof}
        We construct such a function by symmetrizing the indicator of a spherical cap. Let $C \subset S^{n-1}$ be a spherical cap centered at the north pole $e_1$ with angular radius $\theta$, such that its closure is disjoint from the region of thresholding.
        Specifically, since $K \subseteq [-1, 1)$ is compact, there exists some $r > 0$ such that $\sup K < \cos(2r)$. Let $1_C$ be the indicator function of a spherical cap of radius $r$. Consider the autocorrelation kernel:
        $$ K(x, y) = \int_{O(n)} 1_C(U x) 1_C(U y) \, d\mu(U) $$
        where $\mu$ is the Haar measure on the orthogonal group $O(n)$. This kernel is positive definite by construction. The value $K(x,y)$ is proportional to the volume of the intersection of two caps of radius $r$ separated by the angle between $x$ and $y$.
        This intersection volume is non-zero if and only if the { spherical} distance between $x$ and $y$ is less than $2r$. In terms of inner products, the support of the function $f(t)$ corresponding to this kernel is $[\cos(2r), 1]$. Since $\sup K < \cos(2r)$, the function $f$ vanishes identically on $K$.
    \end{proof}
    {
    \begin{remark}
        We note that one may obtain a function vanishing on exactly $K,$ if desired. Specifically, taking the kernel arising from the union of two small caps with some angle $\theta$ gives us a function supported near $0$ and $\cos \theta.$ One might be compelled to find such a function to avoid thresholding your function to $0$ away from $K.$ However, as we are perturbing the rest of the values, the qualitative improvement may be minimal.
    \end{remark}}

    Given the existence of thresholding functions, a natural question is how to keep score. A natural measure is the weight of the linear coefficient in the Gegenbauer expansion. We { define} the {\bf faithfulness constant} for a compact set $K$ in the $n$-dimensional sphere, to be the maximum possible linear coefficient of a unital continuous positive definite function on $S^{n-1}$ vanishing on $K$; { and we denote it by} $\tau_{K,n}.$
    \begin{theorem}\label{degrade}
        Let $f$ be a positive definite function on $S^{n-1}$ with normalized Gegenbauer coefficients $a_k$.
        Let $\iota: S^{n-1} \rightarrow \mathcal{H}$ be the corresponding kernel embedding.
        The norm of the linear map $T_1: \mathcal{H} \to \mathbb{R}^n$ such that $T_{1}\iota(x) = x$ is equal to $1/\sqrt{a_1}$.
    \end{theorem}
    \begin{proof}
        Write the space of spherical harmonics as a direct sum of orthogonal subspaces $\mathcal{H}_k$, each corresponding to the span of spherical harmonics of { total} degree $k$, weighted by $a_k$:
        $$ \mathcal{H}_f = \bigoplus_{k=0}^\infty \mathcal{H}_k.$$
        The embedding $\iota(x)$ can be written as $\sum_{k=0}^\infty v_k(x)$, where $v_k(x) \in \mathcal{H}_k$.
        Due to the orthogonality of the Gegenbauer expansion (and consequently the spherical harmonics), the linear subspace $\mathcal{H}_1$ is orthogonal to all $\mathcal{H}_k$ for $k \neq 1$. Thus, the norm of the recovery map $T_1$
        is exactly $1/\sqrt{a_1}.$
    \end{proof}

    Given the constraint that the optimal thresholding function must maximize this linear coefficient, we can establish a strict structural ceiling on its behavior. The linear coefficient for an optimizer has certain size constraints, for example $a_1\geq a_0.$ By viewing the coefficients as a probability measure on the natural numbers, we can extract this bound directly from the three-term recurrence of the orthogonal polynomials.
    
   \begin{theorem}\label{lem:difference_bound}
        Let $f(t) = \sum_{k=0}^\infty a_k \Gegen{k}(t)$ be a continuous positive definite function on $S^{n-1}$ with $f(1) = 1$ that maximizes the faithfulness constant $a_1$ subject to vanishing on a compact set $K \subset [-1, 1)$. Then the coefficients of $f$ satisfy the second-order difference inequality:
        \[ c_{k-1} a_{k-1} + b_{k+1} a_{k+1} \leq a_1 \frac{d_k}{n} \quad \text{for all } k \geq 0, \]
        where $a_{-1} = c_{-1} = 0$, $d_k$ is the dimension of the space of spherical harmonics of degree $k$, and $c_k, b_k$ are the transition weights from the three-term recurrence $t \Gegen{k}(t) = c_k \Gegen{k+1}(t) + b_k \Gegen{k-1}(t)$.
    \end{theorem}
    \begin{proof}
        Let $g(t) = \Gegen{m}(t)$. By Schoenberg's theorem, $g$ is positive definite. Consider the pointwise product $h(t) = f(t)g(t)$. By the Schur product theorem, $h$ is positive definite. Furthermore, since $f(t) = 0$ for all $t \in K$, $h(t)$ also vanishes on $K$. Since $h(1) = f(1)g(1) = 1$, the function $h$ is an admissible thresholding function. Because $f$ maximizes the linear coefficient among all such functions, the linear coefficient of $h$, denoted $c_1(h)$, must satisfy $c_1(h) \leq a_1$.
        
        To compute $c_1(h)$, we project $h(t)$ onto $\Gegen{1}(t) = t$ using the orthogonality of the normalized Gegenbauer polynomials. Under the rotationally invariant probability measure $\sigma$ on $S^{n-1}$, the squared $L^2$-norm of $\Gegen{k}$ is $\|\Gegen{k}\|^2 = 1/d_k$. 
        
        The linear coefficient is given by:
        \[ c_1(h) = \frac{\langle f(t)\Gegen{m}(t), t \rangle}{\| \Gegen{1} \|^2}. \]
        
        Applying the asymmetric three-term recurrence term by term in the series for $f$ and applying orthogonality:
        \begin{align*} c_1(h) &= n \left( c_m a_{m+1} \|\Gegen{m+1}\|^2 + b_m a_{m-1} \|\Gegen{m-1}\|^2 \right) \\ &= n \left( \frac{c_m a_{m+1}}{d_{m+1}} + \frac{b_m a_{m-1}}{d_{m-1}} \right). \end{align*}
        
        The dimensions and transition weights of the sphere satisfy the detailed balance condition $d_m c_m = d_{m+1} b_{m+1}$. Applying this symmetry, we substitute $c_m/d_{m+1} = b_{m+1}/d_m$ and $b_m/d_{m-1} = c_{m-1}/d_m$ to obtain:
        \[ c_1(h) = \frac{n}{d_m} \left( b_{m+1} a_{m+1} + c_{m-1} a_{m-1} \right). \]
        
        Enforcing the variational bound $c_1(h) \leq a_1$ yields the desired inequality:
        \[ c_{m-1} a_{m-1} + b_{m+1} a_{m+1} \leq a_1 \frac{d_m}{n}. \]
        Since $m \geq 0$ was chosen arbitrarily, the bound holds for all $k \geq 0$.
    \end{proof}
    \begin{remark}
        The presence of the dimension ratio $d_k/n$ on the right side of the inequality is mathematically profound. Because $d_k$ grows as $\mathcal{O}(k^{n-2})$, the structural bound is comically mild. It acts merely as a local anti-spiking condition at low frequencies and becomes practically vacuous at high frequencies. It is far too weak to independently force the sequence to taper or truncate.

        Despite the sheer leniency of this structural bound, the algebraic geometry of the sphere enforces a much stricter reality for simple thresholding sets. As established in \cite{damase2026complete}, when the thresholding set $K$ is finite, the optimal positive definite function { can be chosen} to be a polynomial. 
        { Additionally, if we assume that $K$ is symmetric, then the optimizer is guaranteed to be a polynomial!}
    \end{remark}

    \begin{remark}
        The analogous bound for the classical Delsarte method
        is $a_k\leq d_k a_0.$ In $\mathcal{H}_k\oplus \mathcal{H}_0,$ one can have at most $d_k$
        vectors that are orthogonal with positive projection onto $\mathcal{H}_0.$ Whether the bound is attained or a Delsarte hallucination \cite{damase2026complete} is a difficult geometric problem. Given that analogy, the ratio ${d_k}/{n}$ in Theorem \ref{lem:difference_bound} measures the relative sizes of $\mathcal{H}_k$ and $\mathcal{H}_1.$
    \end{remark}

\section{Explicit faithfulness bounds and the two-point gap}
    We now discuss thresholding for particular choices of $K.$ { Recall the quantity $\tau_{K,n}$ defined prior to Theorem \ref{degrade}.}
    \subsection{One Point Thresholding}
    \begin{theorem}
        Let $K=\{\varepsilon\}$ with $\varepsilon > 0$ small. Then:
        $$\tau_{K,n} \geq \frac{|\Gegen{2}(\varepsilon)|}{\varepsilon + |\Gegen{2}(\varepsilon)|}.$$
        Specifically, $\tau_{K,n} \to 1$ as $\varepsilon \to 0$.
    \end{theorem}
    \begin{proof}
        We construct a function using only the first and second normalized Gegenbauer polynomials: $f(t) = a_1 \Gegen{1}(t) + a_2 \Gegen{2}(t)$.
        To be a valid correlation function, we require $f(1) = a_1 + a_2 = 1$, which implies $a_2 = 1 - a_1$.
        The condition $f(\varepsilon) = 0$ yields:
        $$ a_1 \varepsilon + (1-a_1) \Gegen{2}(\varepsilon) = 0 $$
        Solving for $a_1$:
        $$ a_1 (\varepsilon - \Gegen{2}(\varepsilon)) = -\Gegen{2}(\varepsilon) \implies a_1 = \frac{-\Gegen{2}(\varepsilon)}{\varepsilon - \Gegen{2}(\varepsilon)} $$
        Since $\Gegen{2}(0) = -1/(n-1) < 0$, for sufficiently small $\varepsilon$, $\Gegen{2}(\varepsilon)$ is negative, yielding a valid $a_1 \in (0,1)$.
    \end{proof}

    \begin{figure}
        \centering
        \includegraphics[width=0.75\linewidth]{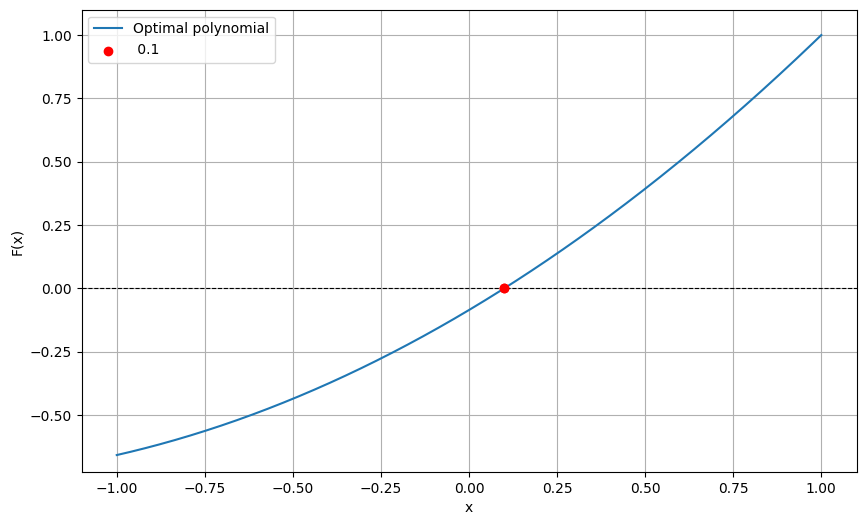}
        \caption{The thresholding function for $K = \{0.1\}, n=3$}
        \label{fig:placeholder}
    \end{figure}

    \subsection{Two Point Thresholding}
    The penalty for thresholding for more than one point is in general severe.
            \begin{figure}
        \centering
        \includegraphics[width=0.75\linewidth]{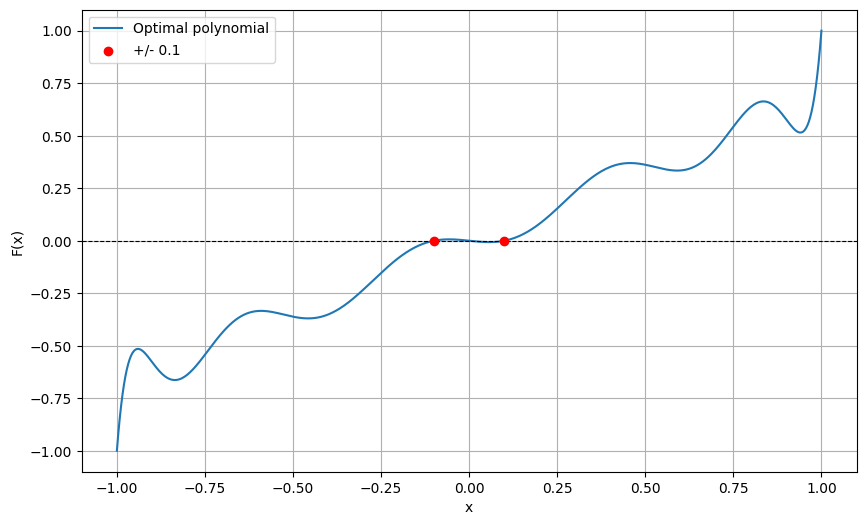}
        \caption{The thresholding function for $K = \{\pm 0.1\}, n=3$}
        \label{fig:placeholder}
    \end{figure}
    \begin{theorem} \label{twopt}
        Let $K=\left\{- \varepsilon,\varepsilon\right\}$ with $1>\varepsilon >0$. Then,
        $$\tau_{K,n} = \frac{\Sigma}{\varepsilon + \Sigma}, \quad \text{where } \Sigma = \sup_{k} \Gegen{2k-1}(-\varepsilon). $$
    \end{theorem}
    \begin{proof}
        Consider $f(t) = \sum a_k \Gegen{k}(t)$ be a positive definite function such that $f(x) = 0$ when $x \in K.$ 
        So, $f(\pm \varepsilon)=0.$ Without loss of generality, we may take $f$ odd.
        Taking $k$ to be the maximizer for $\Gegen{2k-1}(-\varepsilon)$ and choosing
        $$f(x) = \frac{\Gegen{2k-1}(-\varepsilon) x + \varepsilon \Gegen{2k-1}(x)}{\varepsilon+\Gegen{2k-1}(-\varepsilon)}$$
        gives the appropriate $f$ to witness faithfulness.
    \end{proof}
    If $n=2,3$ our bound approaches $1$ as $\varepsilon$ goes to $0.$ If $n \geq 4$ one obtains $\frac{3}{n+2}.$
    (For $n$ large enough, the optimal function is approximately $x^3$ for small $\varepsilon.$)

    The results above highlight a sharp dichotomy. Thresholding a single value near zero allows $a_1 \approx 1$. However, thresholding an interval (or even just two points $\pm \varepsilon$) forces $a_1$ to be small (bounded by a factor roughly proportional to $1/n$).

    \subsection{Interval thresholding}
    \begin{theorem}\label{intthm}
        Let $K=[- \varepsilon,\varepsilon]$ with $1>\varepsilon>0.$ Then,
        $$ \lim_{\varepsilon \to 0} \tau_{K,n} \leq \frac{\Sigma}{1 + \Sigma}, \quad \text{where } \Sigma = \sup_{k \text{ odd}, k \geq 3} \left| \frac{d}{dt}\Gegen{k}(0) \right|. $$
    \end{theorem}
    \begin{proof}
        If $n=2,3$ the supremum is infinite and thus the claim is trivial. We will assume $n\geq 4.$
        Let $f(t) = \sum a_k \Gegen{k}(t)$ be a positive definite function such that $f(x) = 0$ when $x \in K.$ Note $f'(0)=0.$
        Also, since $f(0)$ must vanish, we assume $f$ is odd. Differentiating the expansion term-wise:
        $$ f'(0) = \sum_{k=0}^\infty a_k \frac{d}{dt}\Gegen{k}(0) = 0. $$
        (Differentiation under the sum is valid as Darboux estimates imply the derivative sum converges uniformly to a continuous function
        on $K.$)
        We isolate the linear term $k=1$, noting that $\frac{d}{dt}\Gegen{1}(0) = 1$:
        $$ a_1 \cdot 1 + \sum_{k \neq 1} a_k \frac{d}{dt}\Gegen{k}(0) = 0. $$
        Since $\Gegen{k}$ are even for even $k$, their derivatives at 0 vanish. The sum restricts to odd $k \geq 3$.
        $$ a_1 = - \sum_{k \text{ odd}, k \geq 3} a_k \frac{d}{dt}\Gegen{k}(0) = \sum_{k \text{ odd}, k \geq 3} a_k \left| \frac{d}{dt}\Gegen{k}(0) \right|. $$
        We maximize $a_1$ subject to $\sum a_k = 1$. Let $S_k = |\frac{d}{dt}\Gegen{k}(0)|$. Then $a_1 = \sum_{k \geq 3} a_k S_k$.
        Using $\sum_{k \geq 3} a_k = 1 - a_1$, we have $a_1 \leq (1-a_1) \sup S_k$.
        Rearranging yields $a_1 \leq \frac{\sup S_k}{1 + \sup S_k}$.
    \end{proof}
    One may also see the above as a corollary of Theorem \ref{twopt} via L'Hôpital's rule.

    %
%

  \section{Conjectures and remarks}
    The following conjecture was made in \cite{damase2026complete} in terms of the language of Class I representations of homogeneous spaces.
    \begin{conjecture}
                Any positive definite function on $S^{n-1}$ is a pointwise limit of continuous ones.
    \end{conjecture}
    One may, of course, run Delsarte's method or our linear Delsarte method for soft thresholding
    through a discontinous positive definite function just as easily as a continuous one.
    We have the following paradox: supposing the conjecture were false, we would likely obtain
    bounds sharper than those attained by Delsarte's method for particular finite sets $K.$
    Dixmier shows that any positive definite function which is measurable is a pointwise limit of continuous ones by a convolution argument \cite{dixmier},
    and thus any such pathological function is inherently nonmeasurable, and perhaps axiom dependent.
    That is, for $K$ with some particular constellation of entries one sort of expects large unexplained drops in
    the Delsarte type bounds arising from discontinuous representations.

\subsection{Clustering and Sparsity}
       In light of Theorems \ref{intthm} and \ref{twopt}, we see that soft thresholding a set with two or more points must sharply dampen the off diagonal entries of covariance matrices. That is, the payment for enforcing a prior that random variables are consistent in a geometrically unbiased way (a kernel embedding arising from a positive definite function) is extortionate.
       
       In the statistical literature, thresholding estimators are typically justified by assuming the true population matrix is highly sparse or banded \cite{bickel2008covariance}. Our results provide a rigorous geometric justification for why such assumptions are unavoidable: without an inherent clustered or sparse structure, preserving positive definiteness under thresholding violently collapses the signal.

    However, embeddings of distributions (e.g. $\sum_{x} \omega_x\iota(x)$ where $\omega_x \geq 0$ and $1=\sum_{x}\omega_x$)
    may have meaningfully large correlations.
    Such is a natural, as our methods are naturally suited to low sample high feature data, (where the covariance matrix is low rank) which inherently requires clustering, choosing representative features (e.g. LASSO type methods) or other treatment to be seriously rigorous.
    Finding appropriately meaningful distributions seems to be some kind of clustering
    type problem.

    \subsection*{Acknowledgements}

We thank Apoorva Khare for numerous comments and suggestions.

    \bibliographystyle{plain}
    \bibliography{references}

\end{document}